\documentclass{amsart}
\usepackage{setspace}
\usepackage{a4}
\usepackage{amssymb,amsmath,amsthm,latexsym}
\usepackage{amsfonts}
\usepackage{amsfonts}
\usepackage{graphicx}
\usepackage{textcomp}
\usepackage{cite}
\newtheorem{theorem}{Theorem}[section]

\newtheorem{corollary}[theorem] {Corollary}

\newtheorem{lemma} [theorem]{Lemma}

\newtheorem{proposition}[theorem]{Proposition}

\title{This is the title}
\usepackage{amssymb}
\usepackage{amssymb}
\usepackage{amssymb}
\usepackage{amssymb}
\usepackage{amsmath}
\usepackage{tikz}
\usepackage{hyperref}
\usepackage{enumerate}
\usepackage{mathtools}
\usepackage{amsmath}
\usepackage{tikz}
\begin{document}
	\begin{center}
		{\bf\Large{Commutators Close to the Identity in Unital C*-Algebras}}\\
		
		K. Mahesh Krishna and P. Sam Johnson  \\
		Department of Mathematical and Computational Sciences\\ 
		National Institute of Technology Karnataka (NITK), Surathkal\\
		Mangaluru 575 025, India  \\
		Emails: kmaheshak@gmail.com,  sam@nitk.edu.in\\
		
		Date: \today
	\end{center}

\hrule
\vspace{0.5cm}
\author{Author One}
\address{Department of Mathematical and Computational Sciences\\ 
	National Institute of Technology Karnataka (NITK), Surathkal\\
	Mangaluru 575 025, India}
\email{kmaheshak@gmail.com}
\author{Author Two}
\address{Department of Mathematical and Computational Sciences\\ 
	National Institute of Technology Karnataka (NITK), Surathkal\\
	Mangaluru 575 025, India}
\email{sam@nitk.edu.in}

\textbf{Abstract}: Let $\mathcal{H}$ be an infinite dimensional  Hilbert space and $\mathcal{B}(\mathcal{H})$ be the C*-algebra of
all bounded linear operators on $\mathcal{H}$, equipped with the operator-norm. By improving the Brown-Pearcy construction, Terence Tao in 2018, extended the result of Popa [1981] which reads as: 
For each $0<\varepsilon\leq 1/2$,  there exist $D,X \in \mathcal{B}(\mathcal{H})$ with  $\|[D,X]-1_{\mathcal{B}(\mathcal{H})}\|\leq \varepsilon$ 
such that $\|D\|\|X\|=O\left(\log^5\frac{1}{\varepsilon}\right)$, where $[D,X]\coloneqq DX-XD$.  In this paper,  we show that Tao's result still holds for certain class of unital C*-algebras which  include $\mathcal{B}(\mathcal{H})$ as well as the Cuntz algebra $\mathcal{O}_2$.




\section{Introduction}
Let $n\in \mathbb{N}$, $\mathbb{K}$ be scalar field and $M_n(\mathbb{K})$ be the ring of $n$ by $n$ matrices over $\mathbb{K}$. Using the property of trace map
we easily get that there does not exist $D, X   \in M_n(\mathbb{K})$ such that $DX-XD=1_{M_n(\mathbb{K})}$ \cite{HALMOS} (writing matrices as commutator goes as early as 1937 beginning with the work of 
Shoda \cite{SHODA}, see the introduction in the paper \cite{STASINSKI}).  This argument won't work for bounded linear operators 
on infinite dimensional Hilbert space since the map trace is not defined on the algebra $\mathcal{B}(\mathcal{H})$ of all bounded linear operators on an infinite dimensional Hilbert space $\mathcal{H}$ (it is defined for a proper subalgebra of $\mathcal{B}(\mathcal{H})$  known as the trace class operators \cite{SCHATTEN}). Using 
the property of spectrum of bounded linear operator,  Winter in 1947  proved that the equation $[D,X]\coloneqq DX-XD=1_{\mathcal{B}(\mathcal{H})}$ fails to exist in $\mathcal{B}(\mathcal{H})$ \cite{WINTNER}. After two years, 
Wielandt \cite{WIELANDT} (also see \cite{PUTNAM}) gave a 
simple proof for the failure of this equation. Now it is natural to ask whether we can find $D,X \in \mathcal{B}(\mathcal{H})$ such that the commutator $[D,X]$ is close to the identity operator, in operator-norm. This was 
answered by Brown and Pearcy in 1965 showing that $[D,X]$ can be made close to the identity operator \cite{BROWNPEARCY}. Brown and Pearcy also characterized class of commutators of operators. Following the paper \cite{BROWNPEARCY} of Brown and Pearcy there is a series of papers devoted to the  study of commutators
on   sequence spaces, $\mathcal{L}^p$-spaces,  Banach spaces, C*-algberas, von Neumann algebras,   Banach *-algebras etc. However, a quantitative study of commutators close to the identity operator remains untouched. We start with the 
 following quantitative bound given by Popa in 1981 for product of norm of operators whenever the commutator is close to the identity \cite{POPA}.
\begin{theorem}\cite{POPA}\label{POPAFIRST}
  	Let $\mathcal{H}$ be an infinite dimensional  Hilbert space. Let 	$D,X \in \mathcal{B}(\mathcal{H})$ be such that 
  		\begin{align*}
  	\|[D,X]-1_{\mathcal{B}(\mathcal{H})}\|\leq \varepsilon
  	\end{align*} 
  	for some $\varepsilon>0$. Then  
  		\begin{align*}
  	\|D\|\|X\|\geq\frac{1}{2}\log\frac{1}{\varepsilon}.
  	\end{align*}
  \end{theorem}
Now the problem in  Theorem \ref{POPAFIRST}, the problem  is the existence of $D,X \in \mathcal{B}(\mathcal{H})$ such that the commutator $[D,X]$ is close to the identity operator. This was again obtained by Popa which is stated in the following result. Given real $r $ and positive $s$, by $r=O(s)$ we mean that there is positive $\gamma$ such that $|r|\leq \gamma s$.
  \begin{theorem}\cite{POPA, TAO}\label{POPASECOND}
  		Let $\mathcal{H}$ be an infinite dimensional  Hilbert space.
  	Then 	for each $0<\varepsilon\leq 1$,  there exist $D,X \in \mathcal{B}(\mathcal{H})$ with 
  	\begin{align*}
  	\|[D,X]-1_{\mathcal{B}(\mathcal{H})}\|\leq \varepsilon
  	\end{align*} 
  	and
  	\begin{align*}
  	\|D\|\|X\|=O(\varepsilon^{-2}).
  	\end{align*}
  \end{theorem}
Recently Terence Tao improved Theorem \ref{POPASECOND} and obtained the following theorem.
  \begin{theorem}\cite{TAO}\label{TAOTHEOREM}
  	Let $\mathcal{H}$ be an infinite dimensional  Hilbert space.
  Then 	for each $0<\varepsilon\leq 1/2$,  there exist $D,X \in \mathcal{B}(\mathcal{H})$ with 
  \begin{align*}
  \|[D,X]-1_{\mathcal{B}(\mathcal{H})}\|\leq \varepsilon
  \end{align*} 
  	such that 
  	\begin{align*}
  	\|D\|\|X\|=O\left(\log^5\frac{1}{\varepsilon}\right).
  	\end{align*}
  \end{theorem}
  In this paper, following the arguments in \cite{TAO}, we show that Theorem \ref{TAOTHEOREM} remains valied not only for $\mathcal{B}(\mathcal{H})$ but for certain other classes of unital C*-algebras such as any algebra containing Cuntz algebra $\mathcal{O}_2$. Throughout the papar, we move along the lines of the paper \cite{TAO}. In the rest of introduction, we recall fundamentals of matrices over unital C*-algebras. For more information, we refer \cite{WEGGEOLSEN, MURPHY}.

  Let  $\mathcal{A}$ be a unital C*-algebra. For $n\in \mathbb{N}$, $M_n(\mathcal{A})$ is defined as the set of all $n$ by $n$ matrices over $\mathcal{A}$. It is clearly an algebra with respect to natural matrix operations. We define the involution of an element  $A\coloneqq (a_{i,j})_{1\leq i,j \leq n}\in M_n(\mathcal{A})$ as $A^*\coloneqq (a_{i,j}^*)_{1\leq i,j \leq n}$. Then $M_n(\mathcal{A})$ is a *-algebra. From the Gelfand-Naimark-Segal theorem there exists unique universal representation $(\mathcal{H}, \pi) $, where $\mathcal{H}$ is a Hilbert space, $\pi:M_n(\mathcal{A})\to M_n(\mathcal{B}(\mathcal{H}))$ is an isometric *-homomorphism. This gives a norm on $M_n(\mathcal{A})$ defined as 
  \begin{align*}
  \|A\|\coloneqq \|\pi(A)\|, \quad \forall A \in M_n(\mathcal{A}).
  \end{align*}
  This norm makes $M_n(\mathcal{A})$ as a C*-algebra.

 \section{Commutators close to the identity in unital C*-algebras}
 
In the sequel, $\mathcal{A}$ is a unital C*-algebra. We first derive a lemma followed by a corollary for unital C*-algebras. We omit the proof as it involves an algebraic computation alike the ones given in   \cite{TAO}. 
 \begin{lemma}(Commutator calculation)\label{COMMUTATORLEMMA}
 Let $u,v, b_1, \dots, b_n \in \mathcal{A}	$ and $\delta>0$. Let 
 \begin{align*}
 X\coloneqq  \begin{pmatrix}
 0 & 0 & 0& \cdots & 0 &\delta b_1\\
 1_\mathcal{A} & 0 & 0& \cdots& 0 &\delta b_2\\
 0 & 1_\mathcal{A} & 0& \cdots& 0 &\delta b_3\\
 \vdots &\vdots  & \vdots& \cdots& \vdots &\vdots\\
 0&0&0&\cdots& 0& \delta b_{n-1}\\
 0&0&0&\cdots& 1_\mathcal{A}& \delta b_n
 \end{pmatrix} \in M_n(\mathcal{A})
 \end{align*}
 and 
 \begin{align*}
 D\coloneqq  \begin{pmatrix}
 \frac{v}{\delta} &1_\mathcal{A} & 0& \cdots& 0 & \delta b_1u\\
 \frac{u}{\delta} & \frac{v}{\delta} & 2.1_\mathcal{A}& \cdots& 0 & \delta b_2u\\
 0 & \frac{u}{\delta} & \frac{v}{\delta}& \cdots& 0 & \delta b_3u\\
 \vdots &\vdots  & \vdots& \cdots& \vdots &\vdots\\
 0&0&0&\cdots& \frac{v}{\delta}&  (n-1)1_\mathcal{A}+\delta b_{n-1}u\\
 0&0&0&\cdots& \frac{u}{\delta}&  \frac{v}{\delta}+\delta b_nu
 \end{pmatrix} \in M_n(\mathcal{A}).
 \end{align*}
 Then 
 \begin{align*}
 [D,X]=1_{M_n(\mathcal{A})}+\begin{pmatrix}
 0 & 0 & 0& \cdots& 0 & [v,b_1]+0+\delta b_2+\delta b_1[u,b_n]\\
 0 & 0 & 0& \cdots& 0 &[v,b_2]+[u,b_1]+2\delta b_3+\delta b_2[u,b_n]\\
 0 & 0 & 0& \cdots& 0 &[v,b_3]+[u,b_2]+3\delta b_4+\delta b_3[u,b_n]\\
 \vdots &\vdots  & \vdots& \cdots& \vdots &\vdots\\
 0&0&0&\cdots& 0& [v,b_{n-1}]+[u,b_{n-2}]+(n-1)\delta b_{n}+\delta b_{n-1}[u,b_n]\\
 0&0&0&\cdots& 0& [v,b_n]+[u,b_{n-1}]+0+\delta b_n[u,b_n]-n.1_\mathcal{A}
 \end{pmatrix}.
 \end{align*}
 \end{lemma}
\begin{corollary}\label{COROLLARYLEMMA}
 Let $u,v, b_1, \dots, b_n \in \mathcal{A}	$. Assume that for some $\delta>0$, we have equations 
 \begin{align}\label{CORASSUMPTION1}
 [v,b_i]+[u,b_{i-1}]+i\delta b_{i+1}+\delta b_i[u,b_n]=0, \quad \forall i =2, \dots, n-1
 \end{align} 
 and
 \begin{align}\label{CORASSUMPTION2}
 [v,b_n]+[u,b_{n-1}]+\delta b_n[u,b_n]=n\cdot 1_{{M_n(\mathcal{A})}}.
 \end{align}
  Then for any $\mu>0$, there exist matrices $D_\mu, X_\mu \in M_n(\mathcal{A})$ such that 
  \begin{align*}
 &\|	D_\mu\| \leq \frac{\|u\|}{\mu^2\delta}+\frac{\|v\|}{\mu\delta}+(n-1)+\delta \sum_{i=1}^{n}\mu^{n-i-1}\|b_i\|\|u\|,\\
  & \|	X_\mu\| \leq 1+ \delta \sum_{i=1}^{n}\mu^{n-i+1}\|b_i\| \quad \text{ and }\\
  &\|[D_\mu,X_\mu]-1_{M_n(\mathcal{A})}\|\leq\mu^{n-1}\|[v,b_1]+\delta b_2+\delta b_1 [u,b_n]\|.
  \end{align*}
\end{corollary}
\begin{proof}
	Let $D$ and $X$ be as in Lemma \ref{COMMUTATORLEMMA}. Define 
	\begin{align*}
	S_\mu \coloneqq \begin{pmatrix}
	\mu^{n-1}& 0 & 0& \cdots& 0 &0\\
	0 & \mu^{n-2} & 0& \cdots& 0 &0\\
	0 &0 & \mu^{n-3}  &\cdots&  0&0\\
	\vdots &\vdots  & \vdots& \cdots& \vdots &\vdots\\
	0&0&0&\cdots& \mu& 0\\
	0&0&0&\cdots& 0& 1
	\end{pmatrix} \in M_n(\mathbb{K}), \quad 
	D_\mu \coloneqq \frac{1}{\mu}S_\mu DS_\mu^{-1}, \quad X_\mu \coloneqq \mu S_\mu XS_\mu^{-1}.
\end{align*}
Then 
\begin{align*}
\|D_\mu \|&=\left\| \begin{pmatrix}
\frac{v}{\mu\delta} &1_\mathcal{A} & 0& \cdots& 0 & \mu^{n-2}\delta b_1u\\
\frac{u}{\mu^2\delta} & \frac{v}{\mu\delta} & 2.1_\mathcal{A}& \cdots& 0 & \mu^{n-3}\delta b_2u\\
0 & \frac{u}{\mu^2\delta} & \frac{v}{\mu\delta}& \cdots& 0 & \mu^{n-4}\delta b_3u\\
\vdots &\vdots  & \vdots& \cdots& \vdots &\vdots\\
0&0&0&\cdots &\frac{v}{\mu\delta}&  (n-1)1_\mathcal{A}+\delta b_{n-1}u\\
0&0&0&\cdots& \frac{u}{\mu^2\delta}&  \frac{v}{\mu\delta}+\mu^{-1}\delta b_nu
\end{pmatrix}\right\|\\
&\leq \left\|\frac{u}{\mu^2\delta}\right\|+\left\|\frac{v}{\mu\delta}\right\|+\|(n-1)1_\mathcal{A}\|+\left\| \begin{pmatrix}
0 &0 & 0& \cdots& 0 & \mu^{n-2}\delta b_1u\\
0 & 0& 0& \cdots& 0 & \mu^{n-3}\delta b_2u\\
0 & 0 & 0& \cdots& 0 & \mu^{n-4}\delta b_3u\\
\vdots &\vdots  & \vdots& \cdots& \vdots &\vdots\\
0&0&0&\cdots& 0&  \delta b_{n-1}u\\
0&0&0&\cdots& 0&  \mu^{-1}\delta b_nu
\end{pmatrix}\right\|\\
&\leq \frac{\|u\|}{\mu^2\delta}+\frac{\|v\|}{\mu\delta}+(n-1)+\delta \sum_{i=1}^{n}\mu^{n-i-1}\|b_i\|\|u\|
\end{align*}
and 
\begin{align*}
\|X_\mu\| &=\left\|\begin{pmatrix}
0 &0 & 0& \cdots& 0 & \mu^{n}\delta b_1\\
1_\mathcal{A} & 0& 0& \cdots& 0 & \mu^{n-1}\delta b_2\\
0 & 1_\mathcal{A} & 0& \cdots& 0 & \mu^{n-2}\delta b_3\\
\vdots &\vdots  & \vdots& \cdots& \vdots &\vdots\\
0&0&0&\cdots& 0&  \mu^2 \delta b_{n-1}\\
0&0&0&\cdots& 1_\mathcal{A}&  \mu\delta b_n
\end{pmatrix}\right\|
\leq \|1_\mathcal{A}\|+\left\|\begin{pmatrix}
0 &0 & 0& \cdots& 0 & \mu^{n}\delta b_1\\
0 & 0& 0& \cdots& 0 & \mu^{n-1}\delta b_2\\
0 & 0 & 0& \cdots& 0 & \mu^{n-2}\delta b_3\\
\vdots &\vdots&   \vdots& \cdots& \vdots &\vdots\\
0&0&0&\cdots& 0&  \mu^2 \delta b_{n-1}\\
0&0&0&\cdots& 0&  \mu\delta b_n
\end{pmatrix}\right\|\\
&\leq 1+ \delta \sum_{i=1}^{n}\mu^{n-i+1}\|b_i\|.
\end{align*}
Now using (\ref{CORASSUMPTION1}) and (\ref{CORASSUMPTION2})	 we get 
\begin{align*}
 \|[D_\mu,X_\mu]-1_{M_n(\mathcal{A})}\|&=\left\|\begin{pmatrix}
 0 &0 & 0& \cdots& 0 & \mu^{n-1}([v,b_1]+\delta b_2+\delta b_1 [v,b_n])\\
 0 & 0& 0& \cdots& 0 & 0\\
 0 & 0 & 0& \cdots& 0 & 0\\
 \vdots &\vdots  & \vdots& \cdots& \vdots &\vdots\\
 0&0&0&\cdots& 0&  0\\
 0&0&0&\cdots& 0&  0
 \end{pmatrix}\right\|\\
 &\leq\mu^{n-1}\|[v,b_1]+\delta b_2+\delta b_1 [u,b_n]\|.
\end{align*}
\end{proof}
Let $\mathcal{A}$ be a unital C*-algebra. Assume that  there are isometries $u, v \in \mathcal{A}$ such that
\begin{align}\label{IMPORTANTEQUATION}
u^*u=v^*v=uu^*+vv^*=1_\mathcal{A} \quad \text{and} \quad  u^*v=v^*u=0.
\end{align}
Examples of such unital C*-algebras are $\mathcal{B}(\mathcal{H})$ ($\mathcal{H}$ is an infinite dimensional Hilbert space) as well as any unital  C*-algebra which contains the Cuntz algebra $\mathcal{O}_2$ (see \cite{CUNTZ} for Cuntz algebra). Note that whenever a unital C*-algebra admits a trace map there are no isometries satisfying Equation (\ref{IMPORTANTEQUATION}). In particular, any finite dimensional unital C*-algebra does not have such elements. It is also clear that no commutative unital C*-algebra can have isometries satisfying Equation (\ref{IMPORTANTEQUATION}).

It is shown in \cite{TAO} that whenever  $\mathcal{H}$ is  an infinite dimensional Hilbert space, then the Banach algebras $\mathcal{B}(\mathcal{H})$ and $M_2(\mathcal{B}(\mathcal{H})) $ are isometrically isomorphic. We now do these results for C*-algebras whenever they have isometries satisfying Equation (\ref{IMPORTANTEQUATION}). To do so we first need a result from the theory of C*-algebras.
\begin{theorem}\cite{TAKESAKI, PEDERSEN}\label{INJECTIOVEHOMOISISO}
\begin{enumerate}[\upshape(i)]
	\item Every *-homomorphism between C*-algebras is norm decreasing.
	\item If a *-homomorphism between C*-algebras is   injective, then it is isometric.
\end{enumerate}
\end{theorem}
\begin{theorem}\label{ALGEBRAMATRIX}
Let  $\mathcal{A}$ be a unital C*-algebra. If  there are isometries $u, v \in \mathcal{A}$ such that Equation (\ref{IMPORTANTEQUATION}) holds, then the map 
\begin{align}\label{FIRSTMAP}
\phi:\mathcal{A} \ni x \mapsto 
  \begin{pmatrix}
u^*xu & u^*xv \\
v^*xu & v^*xv
\end{pmatrix} \in M_2(\mathcal{A})
\end{align}
is a C*-algebra  isomorphism with the inverse map 
\begin{align}\label{SECONDMAP}
\psi:M_2(\mathcal{A})\ni   \begin{pmatrix}
a & b \\
c & d
\end{pmatrix} \mapsto uau^*+ubv^*+vcu^*+vdv^* \in \mathcal{A}.
\end{align}
\end{theorem}
\begin{proof}
Using 	Equation (\ref{IMPORTANTEQUATION}), a direct computation gives 
\begin{align*}
\phi\psi\begin{pmatrix}
a & b \\
c & d
\end{pmatrix}&=\phi(uau^*+ubv^*+vcu^*+vdv^*)\\
&= \begin{pmatrix}
u^*(uau^*+ubv^*+vcu^*+vdv^*)u & u^*(uau^*+ubv^*+vcu^*+vdv^*)v \\
v^*(uau^*+ubv^*+vcu^*+vdv^*)u & v^*(uau^*+ubv^*+vcu^*+vdv^*)v
\end{pmatrix}\\
&=\begin{pmatrix}
1_\mathcal{A} a1_\mathcal{A}+1_\mathcal{A}b0+0c1_\mathcal{A}+0d0 & 1_\mathcal{A}a0+1_\mathcal{A}b1_\mathcal{A}+0c0+0d1_\mathcal{A} \\
0a1_\mathcal{A}+0b0+1_\mathcal{A}c1_\mathcal{A}+1_\mathcal{A}d0& 0a0+0b1_\mathcal{A}+1_\mathcal{A}c0+1_\mathcal{A}d1_\mathcal{A}
\end{pmatrix}\\
&=\begin{pmatrix}
a & b \\
c & d
\end{pmatrix}, \quad \forall \begin{pmatrix}
a & b \\
c & d
\end{pmatrix} \in M_2(\mathcal{A})
\end{align*}
and 
\begin{align*}
\psi\phi x&=\psi \begin{pmatrix}
u^*xu & u^*xv \\
v^*xu & v^*xv
\end{pmatrix}\\
&=u(u^*xu)u^*+u(u^*xv)v^*+v(v^*xu)u^*+v(v^*xv)v^*\\
&=uu^*x(uu^*+vv^*)+vv^*x(uu^*+vv^*)\\
&=uu^*x1_\mathcal{A}+vv^*x1_\mathcal{A}=(uu^*+vv^*)x=x, \quad \forall x \in \mathcal{A}.
\end{align*}
Further, 
\begin{align*}
(\phi(x))^*= \begin{pmatrix}
u^*xu & u^*xv \\
v^*xu & v^*xv
\end{pmatrix}^*=\begin{pmatrix}
u^*x^*u & (v^*xu)^* \\
(u^*xv)^* & v^*x^*v
\end{pmatrix}=\begin{pmatrix}
u^*x^*u & u^*x^*v \\
v^*x^*u & v^*x^*v
\end{pmatrix}=\phi(x^*), \quad \forall x \in \mathcal{A}.
\end{align*}
Hence $\phi$ is a *-isomorphism. Using Theorem \ref{INJECTIOVEHOMOISISO},  to show $\phi$ is a C*-algebra isomorphism (i.e., isometric isomorphism), it suffices to show that $\phi$ is injective. Let $ x \in \mathcal{A}$ be such that $\phi x=0$. Then 
\begin{align*}
u^*xu=u^*xv=0, \quad v^*xv=v^*xu=0.
\end{align*}
Using the first equation we get $uu^*xuu^*=uu^*xvv^*=0$ which implies  $uu^*x=uu^*x(uu^*+vv^*)=0$. Similarly using the second equation we get $vv^*x=0$. Therefore $x=(uu^*+vv^*)x=0$. Hence $\phi$ is injective which completes the proof.
\end{proof}
Along with the lines of Theorem \ref{ALGEBRAMATRIX} we can easily derive the following result.
\begin{theorem}\label{ALLC}
	Let  $\mathcal{A}$ be a unital C*-algebra and $n \in \mathbb{N}$. If  there are isometries $u, v \in \mathcal{A}$ such that Equation (\ref{IMPORTANTEQUATION}) holds, then the map 
	\begin{align*}
	\phi:M_n(\mathcal{A}) \ni X \mapsto 
	\begin{pmatrix}
	u^*Xu & u^*Xv \\
	v^*Xu & v^*Xv
	\end{pmatrix} \in M_{2n}(\mathcal{A})
	\end{align*}
	is a C*-algebra  isomorphism with the inverse map 
	\begin{align*}
	\psi:M_{2n}(\mathcal{A})\ni   \begin{pmatrix}
	A & B \\
	C & D
	\end{pmatrix} \mapsto uAu^*+uBv^*+vCu^*+vDv^* \in M_n(\mathcal{A}), 
	\end{align*}
	where if $X\coloneqq (x_{i,j})_{i,j}$ is a matrix, and $a,b\in \mathcal{A}$, by $aXb$ we mean the matrix $(ax_{i,j}b)_{i,j}$. In particular, the C*-algebras $\mathcal{A}, M_{2}(\mathcal{A}), M_{4}(\mathcal{A}), \dots M_{2n}(\mathcal{A}), \dots $ are all *-isometrically isomorphic.
\end{theorem}
In the rest of the paper, we assume that unital C*-algebra  $\mathcal{A}$ has isometries $u,v$ satisfying Equation (\ref{IMPORTANTEQUATION}). In the next result we use the following notation. Given a vector $x \in \mathcal{A}^n$, $x_i$ means its $i^{\text{th}}$ coordinate.
\begin{proposition}\label{RIGHTPROPOSI}
Let $n\geq2$ and $T:\mathcal{A}^n\to \mathcal{A}^{n-1}$ be the bounded linear operator defined by 
\begin{align*}
T(b_i)_{i=1}^n\coloneqq ([v,b_i]+[u,b_{i-1}])_{i=2}^n, \quad \forall (b_i)_{i=1}^n \in \mathcal{A}^n.
\end{align*}	
Then there exists a bounded linear right-inverse  $R:\mathcal{A}^{n-1}\to \mathcal{A}^{n}$ for $L$ such that 
\begin{align*}
\|Rb\|= \sup_{1\leq i \leq n}\|(Rb)_i\|\leq 8 \sqrt{2}n^2\sup_{2\leq i \leq n}\|b_i\|\leq 8 \sqrt{2}n^2\sup_{1\leq i \leq n}\|b_i\|=8 \sqrt{2}n^2\|b\|, \quad \forall b \in \mathcal{A}^{n}.
\end{align*}
\end{proposition}
\begin{proof}
 Define 
	\begin{align*}
	L:\mathcal{A}^{n-1} \ni (x_i)_{i=2}^n\mapsto \left(-\frac{1}{2}x_iv^*-\frac{1}{2}x_{i+1}u^*\right)_{i=1}^n\in \mathcal{A}^{n}, \quad \text{ where } x_1\coloneqq 0, x_{n+1}\coloneqq 0.
	\end{align*}
	and 
	\begin{align*}
	E:\mathcal{A}^{n-1}\ni (x_i)_{i=2}^n\mapsto \left(\frac{1}{2}(vx_iv^*+vx_{i+1}u^*+ux_{i-1}v^*+ux_iu^*)\right)_{i=1}^n \in \mathcal{A}^{n-1}.
	\end{align*}
	Then 
	\begin{align*}
	TL(x_i)_{i=2}^n&=T\left(-\frac{1}{2}x_iv^*-\frac{1}{2}x_{i+1}u^*\right)_{i=1}^n=-\frac{1}{2}(T(x_iv^*)_{i=2}^n+T(x_{i+1}u^*)_{i=2}^n)\\
	&=-\frac{1}{2}(([v, x_iv^*]+[u, x_{i-1}v^*])_{i=2}^n+([v,x_{i+1}u^*]+[u,x_{i}u^*])_{i=2}^n)\\
	&=-\frac{1}{2}(vx_iv^*-x_iv^*v+ux_{i-1}v^*-x_{i-1}v^*u+vx_{i+1}u^*-x_{i+1}u^*v+ux_{i}u^*-x_{i}u^*u)_{i=2}^n\\
	&=-\frac{1}{2}(vx_iv^*-x_i+ux_{i-1}v^*-x_{i-1}v^*u+vx_{i+1}u^*+ux_{i}u^*-x_{i})_{i=2}^n\\
	&=(x_i)_{i=2}^n-\frac{1}{2}(vx_iv^*+vx_{i+1}u^*+ux_{i-1}v^*+ux_iu^*)_{i=1}^n \\
	&=(1-E)(x_i)_{i=2}^n, \quad \forall (x_i)_{i=2}^n \in \mathcal{A}^{n-1}, \quad \text{ where } 1(x_i)_{i=2}^n\coloneqq (x_i)_{i=2}^n.
	\end{align*}
	We next try to show that the operator $1-E$ is bounded invertible with the help of Neumann series. First step is to change the  norm on $\mathcal{A}^{n}$ to an equivalent norm so that  invertibility property will not  affect in both norms. Define a new norm on  $\mathcal{A}^{n-1}$ by 
	\begin{align*}
	\|(x_i)_{i=2}^n\|'\coloneqq\sup _{2\leq i \leq n}\left(2-\frac{i^2}{n^2}\right)^\frac{-1}{2}\|x_i\|.
	\end{align*}
Let $x=(x_i)_{i=2}^n\in	\mathcal{A}^{n-1}$ be such that $	\|(x_i)_{i=2}^n\|'\leq1$. Then
\begin{align*}
\left(2-\frac{i^2}{n^2}\right)^\frac{-1}{2}\|x_i\|\leq \sup _{2\leq i \leq n}\left(2-\frac{i^2}{n^2}\right)^\frac{-1}{2}\|x_i\|\leq 1, \quad \forall 2\leq i \leq n.
\end{align*} 
 Hence $\|x_i\|\leq  \left(2-\frac{i^2}{n^2}\right)^\frac{1}{2}$ for all $2\leq i \leq n$. Using Theorem \ref{ALGEBRAMATRIX} we now get 
\begin{align*}
\|(Ex)_i\|&=\frac{1}{2}\|vx_iv^*+vx_{i+1}u^*+ux_{i-1}v^*+ux_iu^*\|\\
&=\frac{1}{2} \left\|\begin{pmatrix}
x_{i} & x_{i+1} \\
x_{i-1} & x_{i}
\end{pmatrix}\right\|\leq \frac{1}{2}\left\|\begin{pmatrix}
\|x_{i}\| & \|x_{i+1}\| \\
\|x_{i-1}\| & \|x_{i}\|
\end{pmatrix}\right\|\\
&\leq \frac{1}{2} \left(\|x_{i}\|^2+\|x_{i+1}\|^2+\|x_{i-1}\|^2+\|x_{i}\|^2\right)^\frac{1}{2}\\
&\leq \frac{1}{2} \left(\left(2-\frac{i^2}{n^2}\right)+\left(2-\frac{(i+1)^2}{n^2}\right)+\left(2-\frac{(i-1)^2}{n^2}\right)+\left(2-\frac{i^2}{n^2}\right)\right)^\frac{1}{2}\\
&=  \left(2-\frac{i^2}{n^2}-\frac{1}{2n^2}\right)^\frac{1}{2}\leq \left(1-\frac{1}{8n^2}\right)^\frac{1}{2}\left(2-\frac{i^2}{n^2}\right)^\frac{1}{2}\\
&\leq \left(1-\frac{1}{8n^2}\right)\left(2-\frac{i^2}{n^2}\right)^\frac{1}{2}, \quad \forall 2 \leq i \leq n.
\end{align*}
Hence 
\begin{align*}
\|Ex\|'=\sup _{2\leq i \leq n}\left(2-\frac{i^2}{n^2}\right)^\frac{-1}{2}\|(Ex)_i\|\leq \left(1-\frac{1}{8n^2}\right)\|x\|', \quad \forall x \in \mathcal{A}^{n-1}.
\end{align*}
Since $1-\frac{1}{8n^2}<1$, $1-E$ is invertible and $\|(1-E)^{-1}x\|'\leq 8n^2\|x\|'$. Now going back to the original norm, we get 
\begin{align*}
\frac{1}{\sqrt{2}}\|((1-E)^{-1}x)_i\|&\leq \sup _{2\leq i \leq n}\left(2-\frac{i^2}{n^2}\right)^\frac{-1}{2}\|((1-E)^{-1}x)_i\|\\
&=\|(1-E)^{-1}x\|'\leq 8n^2\|x\|'\\
&= 8n^2\sup _{2\leq i \leq n}\left(2-\frac{i^2}{n^2}\right)^\frac{-1}{2}\|x_i\|\\
&\leq 8n^2\sup _{2\leq i \leq n}\|x_i\|=8n^2\|x\|, \quad \forall x \in \mathcal{A}^{n-1}.
\end{align*}
Define $R\coloneqq L(1-E)^{-1}$. Then $TR=TL(1-E)^{-1}=(1-E)(1-E)^{-1}=1$ and 
\begin{align*}
\|Rb\|&=\sup _{1\leq i \leq n}\|(Rb)_i\|=\|L(1-E)^{-1}b\|\leq \|L\|\|(1-E)^{-1}b\|\leq \|(1-E)^{-1}b\|\\
&=\sup _{2\leq i \leq n}\|((1-E)^{-1}b)_i\|\leq 8\sqrt{2} n^2\|b\|= 8\sqrt{2} n^2\sup _{2\leq i \leq n}\|b_i\|\quad \forall b \in \mathcal{A}^{n-1}.
\end{align*}
\end{proof}
As given in \cite{TAO} we try to shift from the systems of equations (\ref{CORASSUMPTION1}) and  (\ref{CORASSUMPTION2}) to the solution of single equation.  Let $n\geq 2$. Define $a\coloneqq (0, \dots, n)\in \mathcal{A}^n$, 
\begin{align*}
F:\mathcal{A}^n \ni (b_i)_{i=1}^n \mapsto (-2b_3, \dots, -(n-1)b_n,0)\in \mathcal{A}^{n-1}
\end{align*}
and 
\begin{align*}
G:\mathcal{A}^n\times \mathcal{A}^n \ni ((b_i)_{i=1}^n, (c_i)_{i=1}^n)\mapsto (-b_2[u,c_n], \dots, -b_n[u,c_n])\in \mathcal{A}^{n-1}.
\end{align*}
We then have $\|F\|\leq n-1$ and $\|G\|\leq 2$. 
\begin{proposition}
	Systems  (\ref{CORASSUMPTION1}) and  (\ref{CORASSUMPTION2}) have a solution $b$ if and only if 
	\begin{align}\label{SOLUTIONEQUATION}
	Tb=a+\delta F(b)+\delta G(b,b).
	\end{align}
\end{proposition}
\begin{proof}
	Systems  (\ref{CORASSUMPTION1}) and  (\ref{CORASSUMPTION2}) have a solution $b$ if and only if 	
	\begin{align*}
	[v,b_i]+[u,b_{i-1}]=-i\delta b_{i+1}-\delta b_i[u,b_n], \quad \forall i =2, \dots, n-1
	\end{align*} 
	and
	\begin{align*}
	[v,b_n]+[u,b_{n-1}]=-\delta b_n[u,b_n]+n\cdot 1_{{M_n(\mathcal{A})}}
	\end{align*}
	if and only if 
	\begin{align*}
	([v,b_i]+[u,b_{i-1}])_{i=2}^n=(0, \dots, n)+\delta (-2b_3, \dots, -(n-1)b_n,0)+\delta (-b_2[u,b_n], \dots, -b_n[u,b_n])
	\end{align*}
	if and only if 
		\begin{align*}
	Tb=a+\delta F(b)+\delta G(b,b).
	\end{align*}
\end{proof}
The above proposition reduces the work of solving systems (\ref{CORASSUMPTION1}) and  (\ref{CORASSUMPTION2}) to a single operator equation. To solve (\ref{SOLUTIONEQUATION}) we need an abstract lemma from \cite{TAO}.
\begin{lemma}\cite{TAO}\label{TAOTHEOREMABST}
	Let $\mathcal{X}$, $\mathcal{Y}$ be Banach spaces,  $T,F:\mathcal{X}\to \mathcal{Y}$ be bounded linear operators, and  let 
	$G:\mathcal{X}\times\mathcal{X} \to \mathcal{Y}$ be a bounded bilinear operator with bound $r>0$ and let $ a \in\mathcal{Y}$. Suppose that
	$T$ has a bounded linear right inverse $R:\mathcal{Y}\to \mathcal{X}$. If $\delta>0$ is such that 
	\begin{align}\label{LEMMACONDITION}
	\delta(2\|F\|\|R\|+4r\|R\|^2\|a\|)<1,
	\end{align}
	then there exists  $ b \in\mathcal{X}$ with $\|b\|\leq 2 \|R\|\|a\|$ that solves the equation 
	\begin{align*}
	Tb=a+\delta F(b)+\delta G(b,b).
	\end{align*}
\end{lemma}
  \begin{theorem}
  For each $n\geq2$, there exists a solution $b$ to  Equation  (\ref{SOLUTIONEQUATION}) such that $\|b\|\leq 16 \sqrt{2}n^3$. 	
  \end{theorem}
 \begin{proof}
 We apply Lemma 	\ref{TAOTHEOREMABST} for 
 \begin{align*}
 \delta \coloneqq \frac{1}{2000n^5}.
 \end{align*}
 Then using Proposition \ref{RIGHTPROPOSI} , we get 
 \begin{align*}
 \delta(2\|F\|\|R\|+4r\|R\|^2\|a\|)&\leq \frac{1}{2000n^5}(2(n-1)8 \sqrt{2}n^2+4.2.128.n^4.n)\\
 &\leq \frac{1}{2000n^5}(16\sqrt{2}n^3+1024n^5)<1.
 \end{align*}
 Lemma \ref{TAOTHEOREMABST} now says that  there exists a $b$ which satisfies (\ref{SOLUTIONEQUATION}).
 \end{proof} 
 \begin{theorem}\label{BBERF}
For each $n\geq2$, let $b$  be an element satisfying   Equation  (\ref{SOLUTIONEQUATION}) and $\|b\|\leq 16 \sqrt{2}n^3$. Then for $\mu=\frac{1}{2}$, $D_\mu, X_\mu \in M_n(\mathcal{A})$ such that  
\begin{align*}
\|	D_\mu\|=O(n^5), \quad \|	X_\mu\|=O(1), \quad \|[D_\mu,X_\mu]-1_{M_n(\mathcal{A})}\|=O(n^32^{-n}).
\end{align*}
 \end{theorem}
  \begin{proof}
  Let $D_\mu, X_\mu \in M_n(\mathcal{A})$ be as in  Corollary \ref{COROLLARYLEMMA}. We then have  
   \begin{align*}
 &\|	D_\mu\| \leq  4. 2000n^5\|u\|+ 2. 2000n^5\|v\|+(n-1)+\frac{1}{2000n^5} \sum_{i=1}^{n}\frac{1}{2^{n-i-1}}16 \sqrt{2}n^3\|u\|=O(n^5),\\
  & \|	X_\mu\| \leq 1+ \frac{1}{2000n^5} \sum_{i=1}^{n}\frac{1}{2^{n-i-1}}16 \sqrt{2}n^3  =O(1),\\
  &\|[D_\mu,X_\mu]-1_{M_n(\mathcal{A})}\|\leq 2\mu^{n-1}(\|v\|\|b_1\|+\delta \|b_2\|+\delta \|b_1 \|\|u\|\|b_n\|).\\
  &\quad \leq 2\frac{1}{2^{n-1}}(\|u\|16 \sqrt{2}n^3+\frac{1}{2000n^5} 16 \sqrt{2}n^3+\frac{1}{2000n^5} 16 \sqrt{2}n^3 \|u\|16 \sqrt{2}n^3)\\
  &\quad \leq 2\frac{n^3}{2^{n-1}}(\|v\|16 \sqrt{2}+\frac{1}{2000n^5} 16 \sqrt{2}+\frac{1}{2000n^5} 16 \sqrt{2} n^3\|u\|16 \sqrt{2})=O(n^32^{-n}).
  \end{align*}
  \end{proof}
  \begin{theorem}\label{BEFORE}
 Let $0<\varepsilon\leq 1/2$. Then  there exist an  even integer $n$  and   $D,X \in M_n(\mathcal{A})$ with 
 \begin{align*}
 \|[D,X]-1_{M_n(\mathcal{A})}\|\leq \varepsilon
 \end{align*} 
 such that 
 \begin{align*}
 \|D\|\|X\|=O\left(\log^5\frac{1}{\varepsilon}\right).
 \end{align*}
\end{theorem}
  \begin{proof}
 Let $D_\mu, X_\mu \in M_n(\mathcal{A})$ be as in  Corollary \ref{COROLLARYLEMMA}. Theorem \ref{BBERF} says that there are  $\alpha, \beta, \gamma>0$ be such that  	
 \begin{align*}
 \|	D_\mu\|\leq \alpha n^5, \quad \|	X_\mu\|\leq \beta, \quad \|[D_\mu,X_\mu]-1_{M_n(\mathcal{A})}\|\leq \gamma n^32^{-n}.
 \end{align*}
 Since  $2^n>n^4$ all but finitely many $n$'s,  $ \gamma n^32^{-n}<\varepsilon$ all but finitely many $n$'s. 
 We now choose  real $c$ such that $n=c \log\frac{1}{\varepsilon}$ is even $ \gamma n^32^{-n}<\varepsilon$. We then have $\|	D_\mu\|=O(\log^5(\frac{1}{\varepsilon}))$ and $ \|[D_\mu,X_\mu]-1_{M_n(\mathcal{A})}\|\leq \varepsilon$.
  \end{proof}
  Theorem \ref{BEFORE} and Theorem \ref{ALLC} easily give the following.
 \begin{theorem}
 	Let $\mathcal{A}$ be a unital C*-algebra. Suppose  there are isometries $u, v \in \mathcal{A}$ such that Equation (\ref{IMPORTANTEQUATION}) holds. Then 
 	for each $0<\varepsilon\leq 1/2$,  there exist $d,x \in \mathcal{A}$ with 
 	\begin{align*}
 	\|[d,x]-1_{\mathcal{A}}\|\leq \varepsilon
 	\end{align*} 
 	such that 
 	\begin{align*}
 	\|d\|\|x\|=O\left(\log^5\frac{1}{\varepsilon}\right).
 	\end{align*}	
 \end{theorem}

  \section{Acknowledgements}
  We thank Prof. Terence Tao, University of California, Los Angeles, USA for a kind reply which  made us to understand  his paper \cite{TAO}. We also thank Sachin M. Naik for some discussions.
  
 \bibliographystyle{plain}
 \bibliography{reference.bib}

\end{document}